\documentclass[11pt,a4paper]{article}

\usepackage{amsfonts}
\usepackage{amsfonts}
\usepackage{amsfonts}
\usepackage{mathrsfs}
\usepackage{amsmath}
\usepackage{amsfonts}
\usepackage{amssymb}
\usepackage{graphicx}
\usepackage{enumerate}
\usepackage{amsmath}
\usepackage{arydshln}

\allowdisplaybreaks

\begin{document}

\renewcommand{\thefootnote}{\fnsymbol{footnote}}

\newtheorem{theorem}{Theorem}[section]
\newtheorem{corollary}[theorem]{Corollary}
\newtheorem{definition}[theorem]{Definition}
\newtheorem{conjecture}[theorem]{Conjecture}
\newtheorem{question}[theorem]{Question}
\newtheorem{lemma}[theorem]{Lemma}
\newtheorem{proposition}[theorem]{Proposition}
\newtheorem{example}[theorem]{Example}
\newenvironment{proof}{\noindent {\bf
Proof.}}{\rule{3mm}{3mm}\par\medskip}
\newcommand{\remark}{\medskip\par\noindent {\bf Remark.~~}}
\newcommand{\pp}{{\it p.}}
\newcommand{\de}{\em}

\newcommand{\JEC}{{\it Europ. J. Combinatorics},  }
\newcommand{\JCTB}{{\it J. Combin. Theory Ser. B.}, }
\newcommand{\JCT}{{\it J. Combin. Theory}, }
\newcommand{\JGT}{{\it J. Graph Theory}, }
\newcommand{\ComHung}{{\it Combinatorica}, }
\newcommand{\DM}{{\it Discrete Math.}, }
\newcommand{\ARS}{{\it Ars Combin.}, }
\newcommand{\SIAMDM}{{\it SIAM J. Discrete Math.}, }
\newcommand{\SIAMADM}{{\it SIAM J. Algebraic Discrete Methods}, }
\newcommand{\SIAMC}{{\it SIAM J. Comput.}, }
\newcommand{\ConAMS}{{\it Contemp. Math. AMS}, }
\newcommand{\TransAMS}{{\it Trans. Amer. Math. Soc.}, }
\newcommand{\AnDM}{{\it Ann. Discrete Math.}, }
\newcommand{\NBS}{{\it J. Res. Nat. Bur. Standards} {\rm B}, }
\newcommand{\ConNum}{{\it Congr. Numer.}, }
\newcommand{\CJM}{{\it Canad. J. Math.}, }
\newcommand{\JLMS}{{\it J. London Math. Soc.}, }
\newcommand{\PLMS}{{\it Proc. London Math. Soc.}, }
\newcommand{\PAMS}{{\it Proc. Amer. Math. Soc.}, }
\newcommand{\JCMCC}{{\it J. Combin. Math. Combin. Comput.}, }
\newcommand{\GC}{{\it Graphs Combin.}, }
\newcommand{\LAA}{{\it Linear Algeb. Appli.}, }

\title{ \bf The inertia of weighted unicyclic graphs
\thanks{
This work was supported by the Natural Science Foundation of China
(Nos. 11101245, 61202362, 11271256, 11271208), China Postdoctoral
Science Foundation (No. 2013M530869), the Natural Science Foundation
of Shandong (Nos. BS2010SF017, ZR2011AQ005). \newline
$^\ddagger$Corresponding author.} }

\author{Guihai Yu$^{a,b}$,  Xiao-Dong Zhang$^c$$^\ddagger$, Lihua Feng$^d$  \\
{\small a. School of Mathematics, Shandong Institute of Business and Technology}\\
{\small Yantai, Shandong, China, 264005.} \\
{\small b. Center for Combinatorics, Nankai University}\\
{\small Tianjin, China, 300071.} \\
{\small c. Department of Mathematics and MOE-LSC, Shanghai Jiao Tong University}\\
{\small Shanghai, China, 200240.} \\
 {\small d. Department of Mathematics, Central South University} \\
{\small Railway Campus,  Changsha, Hunan, China, 410075 } \\
{\small e-mail: { \tt yuguihai@126.com; xiaodong@sjtu.edu.cn; fenglh@163.com }}\\
}
\date{}
\maketitle
\vspace{-0.5cm}

\begin{abstract}
Let $G_w$ be a weighted graph. The \textit{ inertia} of $G_w$ is
the triple $In(G_w)=\big(i_+(G_w),i_-(G_w), $ $ i_0(G_w)\big)$, where
$i_+(G_w),i_-(G_w),i_0(G_w)$ are the number of the positive, negative
and zero eigenvalues of the adjacency matrix $A(G_w)$ of $G_w$
including their multiplicities, respectively. $i_+(G_w)$, $i_-(G_w)$ is
called the \textit{positive, negative index of inertia}  of $G_w$,
respectively. In
this paper we present a lower bound for the positive, negative index
of weighted unicyclic graphs of order $n$ with fixed girth and
characterize all weighted unicyclic graphs attaining this lower
bound. Moreover, we characterize the weighted unicyclic graphs of
order $n$ with two positive, two negative and at least $n-6$ zero eigenvalues, respectively.
\end{abstract}

 \vskip 0.1cm

\section{Introduction}
\vskip 0.1cm

Let $G$ be a simple graph of order $n$ with vertex set
$V(G)=\{v_1,v_2,\cdots,v_n\}$ and edge set $E(G)$. The
\textit{adjacency matrix} $A(G)=(a_{ij})$ of graph $G$ of order $n$
is a symmetric $(0, 1)$-matrix  such that $a_{ij}=1$ if $v_i$ is
adjacent to $ v_j$ and 0 otherwise. A weighted graph $G_w$ is a pair
$(G,w)$ where $G$ is a simple graph with edge set $E(G)$, called the
{\it underlying graph} of $G_w$, and $w$ is a weight function from
$E(G)$ to the set of nonzero real numbers. The \textit{adjacency
matrix} of $G_w$ on $n$ vertices is defined as the matrix
$A(G_w)=(a_{ij})$ such that $a_{ij}=w(v_iv_j)$ if $v_i$  is adjacent
to $v_j$ and 0 otherwise. The characteristic polynomial of $G_w$ is
the characteristic polynomial of $A(G_w)$, denoted by
$$P_{G_w}(\lambda)=det(\lambda
I-A(G_w))=\lambda^n+a_1^*\lambda^{n-1}+\cdots+a_n^*.$$  The
\textit{inertia} of  $G_w$ is defined to be the triple
$In(G_w)=\big(i_+(G_w), i_-(G_w),i_0(G_w)\big),$ where
$i_+(G_w),i_-(G_w),i_0(G_w)$ are the numbers of the positive,
negative and zero eigenvalues of $A(G_w)$ including multiplicities,
respectively. $i_+(G_w)$ and $i_-(G_w)$ are called the
\textit{positive, negative index of inertia} (abbreviated
\textit{positive, negative index}) of $G_w$, respectively. The
number $i_0(G_w)$ is called the \textit{nullity} of $G_w$. The rank
of $n$-vertex graph $G_w$, denoted by $r(G_w)$, is defined as the
rank of $A(G_w)$. Obviously, $r(G_w)=i_+(G_w)+i_-(G_w)=n-i_0(G_w)$.

%The \textit{diameter} of a weighted graph $G_w$, denoted by
%$D_w(G)$, is the length of the longest path in the underlying graph
%$G$. The \textit{matching number} of $G_w$, denoted by $\beta_w(G)$,
%is the size of a maximal matching of the underlying graph $G$.
 A
graph $G_w$ is called {\it acyclic} (resp. {\it unicyclic,
bipartite}) if its underlying graph $G$ is {\it acyclic} (resp. {\it
unicyclic, bipartite}).
 An \textit{induced subgraph} of $G_w$ is an induced subgraph
of $G$ with the same weights.  For a subgraph $H_w$ of $G_w$, let
$G_w-H_w$ be the subgraph obtained from $G_w$ by deleting all
vertices of $H_w$ and all incident edges. For $V^{\prime}\subseteq
V(G_w)$, $G_w-V^{\prime}$ is the subgraph obtained from $G_w$ by
deleting all vertices in $V^{\prime}$ and all incident edges. A
vertex of a graph $G_w$ is called \textit{pendant} if it has degree
one, and is called \textit{quasi-pendant} if it is adjacent to a
pendant vertex. For a weighted graph $G_w$ on at least two vertices,
a vertex $v\in V(G_w)$ is called {\it unsaturated} in $G_w$ if there
exists a maximum matching $M$ of $G$ in which no edge is incident
with $v$; otherwise, $v$ is called {\it saturated} in $G_w$.

A simple  graph may be regarded as  a weighted graph  in which
 the weight of each edge is $+1$.
A signed graph may  be regarded as a weighted graph in which
 the weight of each edge is $+1$ or $-1$. Moreover, the sign of a signed
cycle, denoted by $sgn(C)$, is defined as the sign of the product of
all edge weights $+1$ or $-1$ on $C$. The signed cycle $C$ is said
to be {\it positive (or negative)} if $sgn(C)=+$ (or $sgn(C)=-$). A
signed graph is said to be {\it balanced} if all its cycles are
positive, otherwise it is called {\it unbalanced}.

%
% Some
%structural parameters of weighted graphs are the same as the ones of
%their underlying graphs such as diameter, matching number, girth.

%The \textit{girth} of the weighted graph is the one of the
%underlying graph.

%Berman and Farber \cite{berman} presented a lower bound for the
%second largest Laplacian eigenvalue of weighted graphs. Das and
%Bapat \cite{das1}, Das \cite{das2} obtained several sharp upper or
%lower bounds for the spectral radius of weighted graphs. Sorgun and
%B\"{u}y\"{u}kk\"{o}se \cite{sorgun1, sorgun2} derived some sharp
%bounds for the largest (Laplacian) eigenvalue of weighted graphs.
% Gregory et al
%\cite{gregory1} studied the subadditivity of the positive, negative
%indices of inertia and developed certain properties of Hermitian
%rank which were used to characterize the biclique decomposition
%number. Gregory et al \cite{gregory} investigated the inertia of a
%partial join of two graphs and established some relations between
%inertia and biclique decompositions of partial joins of graphs.
%Daugherty \cite{sean} characterized the inertia of unicyclic graphs
%in terms of matching number and obtained a linear-time algorithm for
%computing it. Yu, Feng and Wang \cite{yu feng wang} investigated the
%minimal positive index of inertia among all unweighted bicyclic
%graphs of order $n$ with pendant vertices, and characterized the
%bicyclic graphs with positive index 1 or 2. There is also a large
%body of knowledge regarding to the inertia of unweighted graphs due
%to its many applications in chemistry (see \cite{collatz, fowler,
%gutman, longuet} for details).
The study of eigenvalues of weighted graph has attracted much
attention. Several results about the (Laplacian) spectral radius of
weighted graphs were derived in \cite{berman, das1, das2, sorgun1,
sorgun2}. The inertia of unweighted graphs has attracted some
attention. Gregory et al. \cite{gregory1} studied the subadditivity
of the positive, negative indices of inertia and developed certain
properties of Hermitian rank which were used to characterize the
biclique decomposition number. Gregory et al. \cite{gregory}
investigated the inertia of a partial join of two graphs and
established a few relations between inertia and biclique
decompositions of partial joins of graphs. Daugherty \cite{sean}
characterized the inertia of unicyclic graphs in terms of matching
number and obtained a linear-time algorithm for computing it. Yu et
al. \cite{yu feng wang} investigated the minimal positive index of
inertia among all unweighted bicyclic graphs of order $n$ with
pendant vertices, and characterized the bicyclic graphs with
positive index 1 or 2. Fan et al. \cite{fan1} introduced the nullity
of signed graphs and characterized the unicyclic signed graphs of
order $n$ with nullity $n-2$, $n-3$, $n-4$, $n-5$ respectively. Fan
et al. \cite{fan} characterized the signed graphs of order $n$ with
nullity $n-2$, $n-3$, respectively, and determined the unbalanced
bicyclic signed graphs of order $n$ with nullity $n-3$ or $n-4$ and
bicyclic signed graphs of order $n$ with nullity $n-5$. Sciriha
\cite{sciriha} (see also \cite{chengbo}) characterized the
unweighted graphs of order $n$ with nullity $n-2$ or $n-3$,
respectively. Cheng et al. \cite{cheng1, cheng2} determined the
unweighted graphs of order $n$ with nullity $n-4$ or $n-5$. Guo et
al. \cite{guo} studied some relations between the matching number
and the nullity. The nullity of unweighted graphs has been studied
well in the literature, see \cite{Borovicanin} for a survey.
However, a characterization of unweighted graphs of order $n$ with
nullity at most $n-6$  is still an open problem. There is also a
large body of knowledge related to the inertia of unweighted graphs
due to its many applications in chemistry (see \cite{collatz,
fowler, gutman, longuet} for details). Motivated by the above
description, we shall characterize the weighted unicyclic graphs of
order $n$ with nullity at least $n-6$.

 This paper is organized as follows. In Section 2, some preliminaries are introduced. In Section 3,
 we  present a lower bound for the positive, negative index of weighted unicyclic graphs of order $n$ with girth $k$ $(3\leq k\leq n-2)$ and characterize all
weighted unicyclic graphs attaining this lower bound. Moreover, we
characterize the weighted unicyclic graphs of order $n$ with two
positive (negative) eigenvalues and the weighted unicyclic graphs of
order $n$ with rank $4$, respectively. In Section 4, we determine
the weighted unicyclic graphs with rank 6. In Section 5, we
characterize the weighted unicyclic graphs of order $n$ with rank
$2$, $3$, $5$, respectively.
\section{Preliminaries}

%In order to consider the inertia, we list three types of the
%elementary congruence matrix operations (ECMOs):
\begin{definition} Let $M$ be a Hermitian matrix. The three types of
elementary congruence matrix operations (ECMOs) of $M$ are defined
as follows:

\begin{enumerate}[(1)]
\item interchanging $i$-th and $j$-th rows of $M$, while interchanging $i$-th
and $j$-th columns of $M$;

\item multiplying $i$-th row of $M$ by non-zero number $k$, while
multiplying $i$-th of column of $M$ by k;

\item adding $i$-th row of $M$ multiplied by a non-zero number $k$ to the $j$-th
row, while adding $i$-th column of $M$ multiplied by $k$ to the
$j$-th column.
\end{enumerate}
\end{definition}

\begin{lemma}\cite{horn} (Sylvester's law of inertia)\label{sylvester
law} Let $M$ be an $n\times n$ real symmetric matrix  and $P$ be an
$n\times n$ nonsingular matrix. Then
$$i_+(PMP^{T})=i_+(M);$$
$$i_-(PMP^{T})=i_-(M).$$
\end{lemma}
By Sylvester's law of inertia, ECMOs do not change the inertia of a
Hermitian matrix.

\medskip

Moreover, the  following result is well known.
\begin{lemma}\label{hermitian matrix}
Let $M$ be an $n\times n$ real symmetric matrix and $N$ be the real
matrix obtained by bordering $M$ as follows:
$$N=\left(\begin{array}{cc}M&y\\y^T&a\end{array}\right),$$ where $y$
is a real column vector and $a$ is a real number. Then
\begin{eqnarray*}
i_+(N)-1&\leq& i_+(M)\leq i_+(N),\\
i_-(N)-1&\leq&i_-(M)\leq i_-(N).
\end{eqnarray*}
\end{lemma}

The following result is an immediate consequence of Lemma
\ref{hermitian matrix}.
\begin{lemma}\label{induced subgraph}
Let $H_w$ be an induced subgraph of $G_w$. Then $i_+(H_w)\leq
i_+(G_w)$, $i_-(H_w)\leq i_-(G_w)$.
\end{lemma}

It is known that the following result hold.
%not difficult to verify the following result.

\begin{lemma}\label{bipartite graph}
Let $G_w$ be a weighted bipartite graph. Then $i_+(G_w)=i_-(G_w)$.
\end{lemma}

The following Lemma is essential belong to Theorem 1.1(b) in \cite{gregory}, Lemma 2.3 in \cite{ma haicheng}, or Lemma 2.9 in \cite{yu feng wang}.

\begin{lemma}\label{deleting pendent vertex}%\cite{gregory}
Let $G_w$ be a weighted graph containing a pendant vertex $v$ with
the unique neighbor $u$. Then $i_+(G_w)=i_+(G_w-u-v)+1$,
$i_-(G_w)=i_-(G_w-u-v)+1$ and $i_0(G_w)=i_0(G_w-u-v)$.
\end{lemma}
%\begin{proof}
%Assume that all vertices in $V(G_w)$ are indexed by
%$\{v_1,v_2,\cdots,v_n\}$ with $v_1=v$, $v_2=u$. Let
%$w_{ij}=w(v_iv_j)$. Then
%$$A(G_w)
%=\left(\begin{array}{ccccc}0&w_{12}&0&\cdots&0\\w_{21}&0&w_{23}&\cdots&w_{2n}\\0&w_{32}&0&\cdots&w_{3n}\\
%\vdots&\vdots&\vdots&\ddots&\vdots\\0&w_{n2}&w_{n3}&\cdots&0\end{array}\right),
%$$
%where the first two rows and columns are labeled by $v_1$, $v_2$.
%and
%$$a_{ij}=\left\{
%\begin{array}{l l}
%1, & \quad \mbox{if $v_i\neq v_j$ and $v_i$, $v_j$ are adjacent,}\\
%0, & \quad \mbox{if $v_i=v_j$, or $v_i$, $v_j$ are not adjacent.}
%\end{array} \right.$$
%By applying the ECMOs on $A(G_w)$, we have
%\begin{eqnarray*}
%i_+(G_w)&=&i_+\left(\begin{array}{ccccc}0&w_{12}&0&\cdots&0\\w_{21}&0&0&\cdots&0\\0&0&0&\cdots&w_{3n}\\
%\vdots&\vdots&\vdots&\ddots&\vdots\\0&0&w_{n3}&\cdots&0\end{array}\right)\quad \mbox{by the third ECMO}\\
%&=&i_+\left(\begin{array}{cc}0&w_{12}\\w_{21}&0\end{array}\right)+i_+\left(\begin{array}{ccc}0&\cdots&w_{3n}\\
%\vdots&\ddots&\vdots\\w_{n3}&\cdots&0\end{array}\right)\\
%&=&i_+\left(\begin{array}{cc}0&w_{12}\\w_{21}&0\end{array}\right)+i_+(G_w-\{v_1,v_2\})\\
%&=&1+i_+(G_w-u-v).
%\end{eqnarray*}
%Similarly, $i_-(G_w)=i_-(G_w-u-v)+1$ holds and therefore
%$i_0(G_w)=i_0(G_w-u-v)$ follows.
%\end{proof}

%For convenience, the transformation in Lemma \ref{deleting pendent
%vertex} is called \emph{$\delta$-transformation}.

%The inertia of
%some graphs can be derived by finite steps of
%$\delta$-transformation.

\begin{example}
Let $P_w$ be a weighted path of order $n=2s+t$. Then

$$\Big(i_+(P_w),i_-(P_w),i_0(P_w)\Big)=\left\{
\begin{array}{l l}
\Big(\frac{n}{2},\,\frac{n}{2},\, 0\Big)&\quad \mbox{if $t=0$,}\\[3mm]
\Big(\frac{n-1}{2},\,\frac{n-1}{2},1\Big) &\quad \mbox{if $t=1$.}
\end{array}\right.$$
\end{example}

Let $u,v$ be two pendant vertices of a weighted graph $G_w$. $u,v$
are called \emph{pendant twins} if they have the same neighbor in
$G_w$.
\begin{lemma}\label{pendant twin}
Let $u,v$ be pendant twins of a weighted graph $G_w$. Then
$i_+(G_w)=i_+(G_w-u)=i_+(G_w-v)$, $i_-(G_w)=i_-(G_w-u)=i_-(G_w-v)$.
\end{lemma}
\begin{proof}
Let $u^\prime$ be the common neighbor of $u,v$ with
$w_1=w(uu^\prime)$, $w_2=w(vu^\prime)$. Then the adjacency matrix of
$G_w$ can be expressed as

$$A(G_w) = \left(
\begin{array}{c:c:c}
\begin{matrix}
0 & 0 \\
0 & 0 \\
\end{matrix} &\begin{matrix}
w_1  \\
w_2  \\
\end{matrix}& \text{\Large{\:0}} \\
\hdashline %
\begin{matrix} w_1& w_2 \\
\end{matrix}&0&\alpha\\
\hdashline %
\text{\rule{0pt}{17pt}\Large{$0^t$}}&\alpha^t & B
\end{array}
\right),$$ where $B$ is the adjacency matrix of $G_w-u-v-u^\prime$
and the first three rows and columns are labeled by $u$, $v$ and
$u^\prime$. So we have
\begin{eqnarray*}
i_+(G_w)&=&i_+\left(
\begin{array}{c:c:c}
\begin{matrix}
0 & 0 \\
0 & 0 \\
\end{matrix} &\begin{matrix}
0  \\
w_2  \\
\end{matrix}& \text{\Large{\:0}} \\
\hdashline %
\begin{matrix} 0& w_2 \\
\end{matrix}&0&\alpha\\
\hdashline %
\text{\rule{0pt}{17pt}\Large{$0^t$}}&\alpha^t & B
\end{array}
\right) \quad \mbox{
%\emph
\text{by the third ECMO}}\\[3mm]
&=&i_+\left(
\begin{array}{c:c:c}
0&w_2 & \text{\Large{\:0}} \\
\hdashline %
w_2&0&\alpha\\
\hdashline %
 \text{\rule{0pt}{17pt}\Large{$0^t$}}&\alpha^t & B
\end{array}
\right)\\[3mm]
&=&i_+(G_w-u).
\end{eqnarray*}
Similarly, we have $i_+(G_w)=i_+(G_w-v)$,
$i_-(G_w)=i_-(G_w-u)=i_-(G_w-v)$.
\end{proof}

A subgraph $U$ of an unweighted graph $G$ is called an {\it
elementary subgraph} if each component of $U$ is a single edge or a
cycle. Let $p(U)$, $c(U)$ be the number of components and the number
of cycles contained in an elementary subgraph $U$, respectively.
 \begin{lemma}
  \cite{cvetkovic}\label{coefficients of weighted graphs}
The coefficient of the characteristic polynomial of the weighted
graph $G_w$ can be expressed as $$a_i^*=\sum_{U\in
\mathscr{U}_i}(-1)^{p(U)}2^{c(U)}\prod_{e\in
E(U)}(w(e))^{\zeta(e,U)},$$ where $\mathscr{U}_i$ is the set of all
elementary subgraphs $U$ contained in the underlying graph $G$
having exactly $i$ vertices, $\zeta(e,U)=1$ if $e$ is contained in
some cycle of $U$ and 2 otherwise.
\end{lemma}

%$\zeta(e,U)=\left\{\begin{array}{cc}1&\mbox{if $e$ is contained in
%some cycle of $U$,}\\[3mm]2&\mbox{otherwise.}\end{array}\right.$

The following Lemma for un-weighted graph is  known  (for example see \cite{fiorini}).
\begin{lemma}\label{weighted tree inertia}
Let $T_w$ be a weighted tree of order $n$ with matching number
$m(T_w)$. Then
$$i_+(T_w)= i_-(T_w)=m(T_w), \ \ i_0(T_w)=n-2m(T_w).$$
\end{lemma}
\begin{proof} For completeness, we present a proof.
It is natural that any elementary subgraph in $T$ consists only of
copies of $K_2$ and has an even number of vertices. By Lemma
\ref{coefficients of weighted graphs}, the coefficients of
$P_{T_w}(\lambda)$ with odd subscript are zero. So we only consider
the coefficients with even subscript.

If $i>m(T_w)$, there exists no elementary subgraph and $a^*_{2i}=0$.
Therefore we suppose $0\leq i\leq m(T_w)$ in the sequel. In view of
Lemma \ref{coefficients of weighted graphs}, we have
$a^*_{2i}=(-1)^i\sum_{U\in \mathscr{U}_{2i}}\prod_{e\in
E(U)}(w(e))^2.$ So $a_{2m(T_w)}$ is the last non-zero coefficient of
$P_{T_w}(\lambda)$. It yields that $i_0(T_w)=n-2m(T_w)$. So we have
$i_+(T_w)=i_-(T_w)=m(T_w)$ by Lemma \ref{bipartite graph}.
\end{proof}

\begin{remark}
 Lemma \ref{weighted tree inertia} shows that the inertia of
a weighted tree is independent of the weights.
\end{remark}

\begin{corollary}\label{weighted forest inertia}
Let $G_w$ be a weighted forest of order $n$ with matching number
$m(G_w)$. Then
$$i_+(G_w)= i_-(G_w)=m(G_w), \ \ i_0(G_w)=n-2m(G_w).$$
\end{corollary}

The following result is an extension of Theorem 5.2 in \cite{ma
haicheng}.
\begin{lemma}\label{difference}
Let $G_w$ be a weighted graph. Then $|i_+(G_w)-i_-(G_w)|\leq
c(G_w),$ where $c(G_w)$ is the number of all odd cycles in $G_w$.
\end{lemma}
\begin{proof} By Lemma \ref{bipartite graph}, it suffices to consider the non-bipartite graphs. We apply induction on the number of odd cycles in $G_w$.
Now assume that $G_w$ contains at least one odd cycle and
$|i_+(G_w-v)-i_-(G_w-v)|\leq c(G_w-v)$ holds for a vertex $v$ of
some odd cycle in $G_w$. By Lemma \ref{hermitian matrix}, we have
\begin{eqnarray*}
|i_+(G_w)-i_-(G_w)|&\leq& |i_+(G_w-v)-i_-(G_w-v)|+1\\
&\leq& c(G_w-v)+1\\
&\leq&c(G_w).
\end{eqnarray*}
This completes the proof.
\end{proof}
\section{Inertia of weighted unicyclic graphs}

Let $C_k^w$ be a weighted cycle with vertex set
$\{v_1,v_2,\cdots,v_k\}$ such that $v_i v_{i+1}\in E(C_k^w)$ ($1\leq
i\leq k-1$), $v_1v_k\in E(C_k^w)$. Let $w_i=w(v_iv_{i+1})$ and
$w_k=w(v_kv_1)$. For $C_k^w$, let $W=\prod_{i=1}^k w_i$. For an even
integer $k$, let $W_e=w_2w_4\cdots w_k$ and $W_o=\frac{W}{W_e}$.

 \begin{definition}
 A weighted even cycle
$C_k^w$ is said to be of Type A (resp. Type B) if
$W_o+(-1)^{\frac{k-2}{2}}W_e=0$ (resp.
$W_o+(-1)^{\frac{k-2}{2}}W_e\neq0$).

A weighted odd cycle $C_k^w$ is said to be  of Type C (resp. Type D)
if $(-1)^{\frac{k-1}{2}}W>0$ (resp. $(-1)^{\frac{k-1}{2}}W<0$).
\end{definition}

\begin{lemma}\label{cycle}
Let $C^w_n$ be a weighted cycle of order $n$. Then
$$\Big(i_+(C^w_n),i_-(C^w_n),i_0(C^w_n)\Big) =\left\{\begin{array}{l
l l l}\big(\frac{n-2}{2},\frac{n-2}{2}, 2\big),&\quad
\mbox{if $C_n^w$ is of Type A,}\\[3mm]
\big(\frac{n}{2},\frac{n}{2},0\big),&\quad \mbox{if $C_n^w$ is of Type B,}\\[3mm]
\big(\frac{n+1}{2},\frac{n-1}{2},0\big),&\quad \mbox{if $C_n^w$ is of Type C,}\\[3mm]
\big(\frac{n-1}{2},\frac{n+1}{2},0\big),&\quad \mbox{if $C_n^w$ is
of Type D.}
\end{array}\right. $$
\end{lemma}

\begin{proof} Let $V(C_n^w)=\{v_1,v_2,\cdots,v_n\}$  and $v_iv_{i+1} \in E(C_n^w)$ ($1\leq i\leq
n-1$), $v_1v_n\in E(C_n^w)$. Let $w_i=w(v_iv_{i+1})$ ($1\leq i\leq
n-1$) and $w_n=w(v_nv_1)$. Then
$$A(C^w_n)=\left(\begin{array}{{ccccccc}}0&w_1&0&0&\cdots&0&w_n\\w_1&0&w_2&0&\cdots&0&0\\0&w_2&0&w_3&\cdots&0&0\\ 0&0&w_3&0&\cdots&0&0\\ \vdots&\vdots&\vdots&\vdots&\ddots&\vdots&\vdots\\
0&0&0&0&\cdots&0&w_{n-1}\\w_n&0&0&0&\cdots&w_{n-1}&0\end{array}\right).$$

\textbf{Case 1.} $n$ is even. Applying ECMOs on $A(C^w_n)$, we have
\begin{eqnarray*}
i_+(C^w_n)&=&i_+\left(\begin{array}{ccccccc}0&w_1&0&0&\cdots&0&w_n\\w_1&0&0&0&\cdots&0&0\\0&0&0&w_3&\cdots&0&-\frac{w_2w_n}{w_1}\\ 0&0&w_3&0&\cdots&0&0\\ \vdots&\vdots&\vdots&\vdots&\ddots&\vdots&\vdots\\
0&0&0&0&\cdots&0&w_{n-1}\\w_n&0&-\frac{w_2w_n}{w_1}&0&\cdots&w_{n-1}&0\end{array}\right)\\[2mm]
&=&i_+\left(\begin{array}{ccccccc}0&w_1&0&0&\cdots&0&0\\w_1&0&0&0&\cdots&0&0\\0&0&0&w_3&\cdots&0&-\frac{w_2w_n}{w_1}\\ 0&0&w_3&0&\cdots&0&0\\ \vdots&\vdots&\vdots&\vdots&\ddots&\vdots&\vdots\\
0&0&0&0&\cdots&0&w_{n-1}\\0&0&-\frac{w_2w_n}{w_1}&0&\cdots&w_{n-1}&0\end{array}\right)\\[2mm]
&=&i_+\left(\begin{array}{cc}0&w_1\\w_1&0\end{array}\right)+i_+\left(\begin{array}{ccccc}0&w_3&\cdots&0&-\frac{w_2w_n}{w_1}\\ w_3&0&\cdots&0&0\\ \vdots&\vdots&\ddots&\vdots&\vdots\\
0&0&\cdots&0&w_{n-1}\\-\frac{w_2w_n}{w_1}&0&\cdots&w_{n-1}&0\end{array}\right)\\[2mm]
&=&1+i_+\left(\begin{array}{ccccccc}0&w_3&0&0&\cdots&0&-\frac{w_2w_n}{w_1}\\ w_3&0&0&0&\cdots&0&0\\0&0&0&w_5&\cdots&0&(-1)^2\frac{w_2w_4w_n}{w_1w_3}\\0&0&w_5&0&\cdots&0&0\\ \vdots&\vdots&\vdots&\vdots&\ddots&\vdots&\vdots\\
0&0&0&0&\cdots&0&w_{n-1}\\-\frac{w_2w_n}{w_1}&0&(-1)^2\frac{w_2w_4w_n}{w_1w_3}&0&\cdots&w_{n-1}&0\end{array}\right)\\
\end{eqnarray*}
\begin{eqnarray*}
&=&2+i_+\left(\begin{array}{ccccc}0&w_5&\cdots&0&(-1)^2\frac{w_2w_4w_n}{w_1w_3}\\w_5&0&\cdots&0&0\\ \vdots&\vdots&\ddots&\vdots&\vdots\\
0&0&\cdots&0&w_{n-1}\\(-1)^2\frac{w_2w_4w_n}{w_1w_3}&0&\cdots&w_{n-1}&0\end{array}\right)\\[4mm]
&=& \cdots\\
&=&\frac{n-4}{2}+i_+\left(\begin{array}{cccc}0&w_{n-3}&0&c_1\\w_{n-3}&0&w_{n-2}&0\\0&w_{n-2}&0&w_{n-1}\\c_1&0&w_{n-1}&0\end{array}\right)\quad\quad
\mbox{\big(where $c_1=(-1)^\frac{n-4}{2}\frac{w_2w_4\cdots
w_{n-4}w_n}{w_1w_3\cdots
w_{n-5}}$\big)}\\[3mm]
&=&\frac{n-4}{2}+i_+\left(\begin{array}{cccc}0&w_{n-3}&0&0\\w_{n-3}&0&0&0\\0&0&0&c_2\\0&0&c_2&0\end{array}\right)\quad\quad
\mbox{\big(where $c_2=w_{n-1}+(-1)^\frac{n-2}{2}\frac{w_2w_4\cdots
w_{n-2}w_n}{w_1w_3\cdots
w_{n-3}}$\big)}\\[3mm]
&=&\frac{n-2}{2}+i_+\left(\begin{array}{cc}0&c_2\\
c_2&0\end{array}\right).
\end{eqnarray*}

Moreover, note that

$$i_+\left(\begin{array}{cc}0&c_2\\
c_2&0\end{array}\right)=\left\{\begin{array}{c c}0,&\mbox{if
$c_2=0$,}\\1,&\mbox{if $c_2\neq 0$.}\end{array}\right.$$

Therefore

$$i_+(C^w_n)=\left\{\begin{array}{cc}\frac{n-2}{2},&\quad
\mbox{if $w_1w_3\cdots
w_{n-3}w_{n-1}+(-1)^{\frac{n-2}{2}}w_2w_4\cdots w_{n-2}w_{n}=0$,}\\[3mm]
\frac{n}{2},&\quad \mbox{if $w_1w_3\cdots
w_{n-3}w_{n-1}+(-1)^{\frac{n-2}{2}}w_2w_4\cdots w_{n-2}w_{n}\neq
0$.}\end{array}\right.$$

Similarly, one has

$$i_-(C^w_n)=\left\{\begin{array}{cc}\frac{n-2}{2},&\quad
\mbox{if $w_1w_3\cdots
w_{n-3}w_{n-1}+(-1)^{\frac{n-2}{2}}w_2w_4\cdots w_{n-2}w_{n}=0$,}\\[3mm]
\frac{n}{2},&\quad \mbox{if $w_1w_3\cdots
w_{n-3}w_{n-1}+(-1)^{\frac{n-2}{2}}w_2w_4\cdots w_{n-2}w_{n}\neq
0$.}\end{array}\right.$$

\textbf{Case 2.} $n$ is odd.
By modifying the above procedure, we have
\begin{eqnarray*}
i_+(C^w_n)&=&\frac{n-3}{2}+i_+\left(\begin{array}{ccc}0&w_{n-2}&c_3\\w_{n-2}&0&w_{n-1}\\c_3&w_{n-1}&0\end{array}\right)\quad\quad
\mbox{where $c_3=(-1)^\frac{n-3}{2}\frac{w_2w_4\cdots w_{n-3}w_n}{w_1w_3\cdots w_{n-4}}$}\\[3mm]
\end{eqnarray*}
\begin{eqnarray*}
&=&\frac{n-3}{2}+i_+\left(\begin{array}{ccc}0&w_{n-2}&0\\w_{n-2}&0&0\\0&0&c_4\end{array}\right)\quad\quad
\mbox{where $c_4=2(-1)^\frac{n-1}{2}\frac{w_2w_4\cdots w_{n-1}w_n}{w_1w_3\cdots w_{n-2}}$}\\[3mm]
&=&\frac{n-1}{2}+\left\{\begin{array}{cc}1,&\mbox{if $c_4>0$,}\\0,&\mbox{if $c_4<0$.}\end{array}\right.\\[3mm]
i_-(C^w_n)&=&\frac{n-1}{2}+\left\{\begin{array}{cc}1,&\mbox{if
$c_4<0$,}\\0,&\mbox{if $c_4>0$.}\end{array}\right.
\end{eqnarray*}
It is evident that the sign of $c_4$ is the same as that of
$(-1)^\frac{n-1}{2}\big(\prod_{i=1}^n w_i\big)$, which leads to the
desired result.
\end{proof}

\subsection{Minimal positive (negative) index of weighted unicyclic graphs}

Let $U_{n,k}$ $(n>k)$ be an unweighted unicyclic graph obtained from
a cycle $C_k$ by attaching $n-k$ pendant vertices to a vertex of
$C_k$.
\begin{theorem}\label{thm3-3}
Let $G_w$ be a weighted unicyclic  graph of order $n$ with girth $k$
$(3\leq k\leq n-2)$. Then $i_+(G_w)\geq \lceil\frac{k}{2}\rceil$,
$i_-(G_w)\geq \lceil\frac{k}{2}\rceil$. These bounds are sharp if
 the underlying graph of $G_w$ is $U_{n,k}$.
\end{theorem}

\begin{proof}
It is evident that the underlying graph of $G_w$ must contain
$U_{k+1,k}$ as an induced subgraph. Moreover, by Lemma \ref{deleting
pendent vertex},
$i_+(U_{k+1,k})=i_+(P_{k-1})+1=\lceil\frac{k}{2}\rceil$, so by Lemma
\ref{induced subgraph} $i_+(G_w)\geq \lceil\frac{k}{2}\rceil$.
Similarly, $i_-(G_w)\geq \lceil\frac{k}{2}\rceil$. By Lemmas
\ref{weighted tree inertia} and \ref{deleting pendent vertex}, any
weighted graph with $U_{n,k}$ as its underlying graph has the same
positive (negative) index as $U_{k+1,k}$. Therefore
$i_+(U_{n,k})=\lceil\frac{k}{2}\rceil$ and
$i_-(U_{n,k})=\lceil\frac{k}{2}\rceil$.
\end{proof}

\begin{corollary}\label{maximal nullity}
Let $G_w$ be a weighted unicyclic  graph of order $n$ with girth $k$
$(3\leq k\leq n-2)$. Then $i_0(G_w)\leq n-2\lceil\frac{k}{2}\rceil$.
The bound is sharp if
 the underlying graph of $G_w$ is $U_{n,k}$.
\end{corollary}

\begin{corollary}\label{minimum positive index}
Let $G_w$ be a weighted unicyclic  graph of order $n$ with pendant
vertices. Then $i_+(G_w)\geq 2$, $i_-(G_w)\geq 2$ and $i_0(G_w)\leq
n-4$. These bounds are sharp if
 the underlying graph of $G_w$ is
 $U_{n,3}$ or $U_{n,4}$.

\end{corollary}

\subsection{Weighted unicyclic  graphs with the minimal positive (negative) index}

The following result is an extension of Theorems 3.1 and 3.3 in
\cite{fan2}.

\begin{lemma}\label{matched vertex}
Let $T_w$ be a weighted tree with $u \in V(T_w)$ and $G_w^0$ be a
weighted graph different from $T_w$. Let $G_w$ be a graph obtained from
$G_w^0$ and $T_w$ by joining $u$ with certain vertices of $G_w^0$.
Then the following statements hold:
\begin{enumerate} [(1).]
\item If
$u$ is a saturated vertex in $T_w$, then
$$i_+(G_w)=i_+(T_w)+i_+(G_w^0)=m(T_w)+i_+(G_w^0),$$
$$i_-(G_w)=i_-(T_w)+i_-(G_w^0)=m(T_w)+i_-(G_w^0).$$
\item If $u$ is an
 unsaturated vertex in $T_w$, then
$$i_+(G_w)=i_+(T_w-u)+i_+(G_w^0+u)=m(T_w)+i_+(G_w^0+u),$$
$$i_-(G_w)=i_-(T_w-u)+i_-(G_w^0+u)=m(T_w)+i_-(G_w^0+u),$$
 where
$G_w^0+u$ is the subgraph of $G_w$ induced by the vertices of
$G_w^0$ and $u$.
\end{enumerate}
\end{lemma}

Let $G_w$ be a weighted unicyclic  graph and $C^w_k$ be the unique
weighted cycle of $G_w$. For each vertex $v\in V(C^w_k)$, let
$G_w\{v\}$ be the weighted tree rooted at $v$ and containing  $v$.
Clearly, $G_w\{v\}$ is an induced subgraph of $G_w$. From Lemma
\ref{matched vertex}, it follows that

\begin{lemma}\label{inertia of unicyclic graph}
Let $G_w$ be a weighted unicyclic graph and $C^w_k$ be the unique
weighted cycle of $G_w$. Then the following statements hold:
\begin{enumerate}[(1).]
\item If there exists a vertex $v\in V(C^w_k)$ which is saturated in
$G_w\{v\}$, then $$i_+(G_w)=i_+(G_w\{v\})+i_+(G_w-G_w\{v\}),$$
$$i_-(G_w)=i_-(G_w\{v\})+i_-(G_w-G_w\{v\}).$$

\item If there does not exist a vertex $v\in V(C^w_k)$ which is saturated in
$G_w\{v\}$, then $$i_+(G_w)=i_+(G_w-C^w_k)+i_+(C^w_k),$$
$$i_-(G_w)=i_-(G_w-C^w_k)+i_-(C^w_k).$$
\end{enumerate}
\end{lemma}

Let $G^*$ be an unweighted unicyclic  graph of order $n$ obtained
from a cycle $C_k$ and a star $K_{1,n-k-1}$  of order $n-k$ by
inserting an edge between  a vertex on $C_k$ and the center of
$K_{1,n-k-1}$. Let $\mathcal{U}^w_{n,k}$ $(3\leq k\leq n-2)$ be the
set of weighted unicyclic  graphs of order $n$ with girth $k$. In
the following we shall characterize all weighted unicyclic graphs
with minimal positive (negative) index $\lceil\frac{k}{2}\rceil$
among all graphs in $\mathcal{U}^w_{n,k}$.

%Let $\mathcal {U}_n^*$ be the set of unweighted unicyclic  graph $G$
%of order $n$ which satisfies one of the following conditions:
%\begin{enumerate}[(i)]
%\item $G$ is obtained from $U^\prime$ by joining some appropriate pendant vertices
%to the quasi-pendant vertices in $U^\prime$, where $U^\prime$ is a graph obtained from an even cycle $C_k$ by attaching one pendant
%edge at $t$ $(t\leq \frac{k}{2})$ distinct vertices such that there
%are an odd number of vertices on $C_k$ lying between two consecutive
%vertices of these $t$ vertices;
%\item $G$ is obtained from $U^{\prime\prime}$ by joining some appropriate pendant vertices
%to the quasi-pendant vertices in $U^{\prime\prime}$, where
%$U^{\prime\prime}$ is a graph obtained from an odd cycle $C_k$ by
%attaching one pendant edge at $t$ $(t\leq \frac{k}{2})$ distinct
%vertices such that there are an even number of vertices on $C_k$
%lying between exactly one pair of consecutive vertices of these $t$
%vertices, and there are an odd number of vertices on $C_k$ lying
%between any other pair of consecutive vertices.
%\end{enumerate}

\begin{theorem}\label{extremal graphs1}
Let $G_w\in \mathcal{U}^w_{n,k}$  be a weighted unicyclic
 graph associated  the unique weighted cycle  $C^w_k$
with vertex set $\{v_1,v_2,\cdots,v_k\}$. Then  $i_+(G_w)=\lceil\frac{k}{2}\rceil$ if and only if either (1) or (2) holds.
%\begin{enumerate}
%\item If $k=n$, $G_w$ is one of the following graphs: weighted cycle $C_n$ where $n$ is even and $w_1w_3\cdots
%w_{n-3}w_{n-1}+(-1)^{\frac{n-2}{2}}w_2w_4\cdots w_{n-2}w_{n}\neq 0$;
%weighted cycle $C_n$ where $n$ is odd and
%$(-1)^{\frac{n-1}{2}}\frac{w_2w_4\cdots w_{n-1}w_{n}}{w_1w_3\cdots
%w_{n-4}w_{n-2}}>0$; if $k=n-1$, $G_w$ is a weighted graph with
%$U_{n,n-1}$ as the underlying graph.

(1). If there exists a vertex $v_i\in V(C^w_k)$ which is saturated
in $G_w\{v_i\}$, then $m(G_w\{v_i\})=1$ (equivalently, $G_w\{v_i\}$
is a star) and $m(G_w-G_w\{v_i\})=\lfloor\frac{k-1}{2}\rfloor$.

(2).  If there does not exist a vertex $v\in V(C^w_k)$ which is
saturated in $G_w\{v\}$, then  $G\cong G^*$ and $C_k^w$ is of
$\left\{\begin{array}{c}\mbox{Type A for even k,}\\ \mbox{Type D for
odd k.}\end{array}\right.$.
%\end{enumerate}
\end{theorem}
\begin{proof}
%By Lemmas \ref{induced subgraph} and \ref{cycle}, it is easy to
%verify that the results hold if $k=n$, $n-1$. In the following we
%consider the case $3\leq k\leq n-2$.
%\medskip
Assume that there exists a vertex $v_i\in V(C^w_k)$ which is
saturated in $G_w\{v_i\}$. Without loss of generality, we assume
that $v_i=v_1$. We shall verify the following claim.

\textbf{Claim.} $G_w\{v_1\}$ is a star and
$m(G_w-G_w\{v_1\})=\lfloor\frac{k-1}{2}\rfloor$.
%\left\{\begin{array}{ll}\frac{k-2}{2},&\quad
%\mbox{$k$
%is even,}\\[2mm]
%\frac{k-1}{2},&\quad \mbox{$k$ is odd.}\end{array}\right.$

Since $v_1\in V(C^w_k)$ is saturated in $G_w\{v_1\}$, from Lemmas
\ref{weighted tree inertia} and \ref{inertia of unicyclic graph}, we
have
\begin{eqnarray*}
i_+(G_w)&=&i_+(G_w\{v_1\})+i_+(G_w-G_w\{v_1\})\\
&=&m(G_w\{v_1\})+m(G_w-G_w\{v_1\}).
\end{eqnarray*}

If $k$ is even, then
$$m(G_w\{v_1\})+m(G_w-G_w\{v_1\})=\frac{k}{2}.$$ Note that
$m(G_w\{v_1\})\geq 1$ and $m(G_w-G_w\{v_1\})\geq \frac{k-2}{2}$
since $G_w\{v_1\}$ contains $P_{k-1}$. So it follows that
$m(G_w\{v_1\})=1$ and $m(G_w-G_w\{v_1\})=\frac{k-2}{2}$.

If $k$ is odd, then
$$m(G_w\{v_1\})+m(G_w-G_w\{v_1\})=\frac{k+1}{2}.$$ Note that
$m(G_w\{v_1\})\geq 1$ and $m(G_w-G_w\{v_1\})\geq \frac{k-1}{2}$. So
it follows that $m(G_w\{v_1\})=1$ and
$m(G_w-G_w\{v_1\})=\frac{k-1}{2}$. This completes the proof of
Claim.

\medskip
Now assume that any vertex $v\in V(C^w_k)$ is unsaturated in
$G_w\{v\}$. By Lemmas \ref{weighted tree inertia} and \ref{inertia
of unicyclic graph}, we have
\begin{eqnarray*}
i_+(G_w)&=&i_+(C^w_k)+i_+(G_w-C^w_k)\\
&=&i_+(C^w_k)+m(G_w-C^w_k).
\end{eqnarray*}
Then we have
$$m(G_w-C^w_k)=\left\{\begin{array}{l l l}
1,&\quad \mbox{if $C_k^w$ is of Type $A$,}\\
0,&\quad \mbox{if $C_k^w$ is of Type $B$,}\\
0,&\quad \mbox{if $C_k^w$ is of Type $C$,}\\
1,&\quad \mbox{if $C_k^w$ is of Type $D$.}\\
\end{array}\right.$$

If $m(G_w-C^w_k)=0$, then any vertex not on $C^w_k$ is a pendant
vertex which is adjacent to a vertex on $C^w_k$ in $G_w$. This
contradicts the fact that $v$ is unsaturated in $G_w\{v\}$ for any
$v\in V(C^w_k)$. So $m(G_w-C^w_k)=1$, i.e. $G_w-C^w_k$ is a star.
This implies the result.
\end{proof}

Similar to Theorem \ref{extremal graphs1}, we have
\begin{theorem}\label{extremal graphs2}
Let $G_w\in \mathcal{U}^w_{n,k}$  be a weighted unicyclic
 graph associated the unique weighted cycle $C^w_k$
with vertex set $\{v_1,v_2,\cdots,v_k\}$. Then  $i_-(G_w)=\lceil\frac{k}{2}\rceil$ if and only if either (1) or (2) holds.
%\item If $k=n$, $G_w$ is one of the following graphs: weighted cycle $C_n$ where $n$ is even and $w_1w_3\cdots
%w_{n-3}w_{n-1}+(-1)^{\frac{n-2}{2}}w_2w_4\cdots w_{n-2}w_{n}\neq 0$;
%weighted cycle $C_n$ where $n$ is odd and
%$(-1)^{\frac{n-1}{2}}\frac{w_2w_4\cdots w_{n-1}w_{n}}{w_1w_3\cdots
%w_{n-4}w_{n-2}}<0$; if $k=n-1$, $G_w$ is a weighted graph with
%$U_{n,n-1}$ as the underlying graph.

(1). If there exists a vertex $v_i\in V(C^w_k)$ which is saturated
in $G_w\{v_i\}$, then $m(G_w\{v_i\})=1$ (equivalently, $G_w\{v_i\}$
is a star) and $m(G_w-G_w\{v_i\})=\lfloor\frac{k-1}{2}\rfloor$.

(2).   If there does not exist a vertex $v_i\in V(C^w_k)$ which is saturated in
$G_w\{v_i\}$, then
 $G\cong G^*$ and $C_k^w$ is of $\left\{\begin{array}{c}\mbox{Type A for even k,}\\ \mbox{Type C for
odd k.}\end{array}\right.$.
\end{theorem}

 From Corollary \ref{maximal nullity} and Theorems
\ref{extremal graphs1}, \ref{extremal graphs2}, it follows that
\begin{theorem}\label{extremal graphs3}
Let $G_w\in \mathcal{U}^w_{n,k}$  be a weighted unicyclic
 graph associated the unique weighted cycle $C^w_k$
with vertex set $\{v_1,v_2,\cdots,v_k\}$. Then $i_0(G_w)=n-2\lceil\frac{k}{2}\rceil$ if and only if either (1) or (2) holds.
%\begin{enumerate}

(1).  If there exists a vertex $v_i\in V(C^w_k)$ which is saturated
in $G_w\{v_i\}$, then $m(G_w\{v_i\})=1$ (equivalently, $G_w\{v_i\}$
is a star) and $m(G_w-G_w\{v_i\})=\lfloor\frac{k-1}{2}\rfloor$.

(2).   If there does not exist a vertex $v_i\in V(C^w_k)$ which is saturated in
$G_w\{v_i\}$, then $G\cong G^*$ and $C_k^w$ is of Type A.
%\end{enumerate}
\end{theorem}
\subsection{Weighted unicyclic graphs with two positive (negative)
eigenvalues}

First we define the following four classes of unweighted unicyclic
graphs:
\begin{enumerate}[{}]
\item $U_1^{r,s}$ ($r,s\geq 0$, $r+s=n-3$) is a unicyclic graph of order $n$ obtained by attaching $r$
and $s$ pendant vertices at two different vertices of $C_3$,
respectively, $r$ or $s$ is allowed to be $0$.
\item $U_2^{p,q}$ ($p,q\geq 0$, $p+q=n-4$) is a unicyclic graph of order $n$  obtained  by attaching $p$,
$q$ pendant vertices at two nonadjacent vertices of $C_4$,
respectively, $p$ or $q$ is allowed to be $0$.
\item $U_3^{n-4}$ is a unicyclic graph of order $n$  obtained from $C_3$ and $K_{1,n-4}$ by inserting an edge between
 a vertex of $C_3$ and the center of $K_{1,n-4}$.
\item $U_4^{n-5}$ is a unicyclic graph of order $n$ obtained from $C_4$ and $K_{1, n-5}$ by inserting an edge between a vertex of $C_4$ and the center of $K_{1,n-5}$.
\end{enumerate}

\begin{figure}[ht]
\center
\includegraphics [width = 10cm]{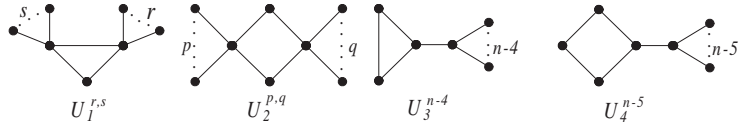}
\caption { \textit{ Four unicyclic graphs $U_1^{r,s}$, $U_2^{p,q}$,
$U_3^{n-4}$, $U_4^{n-5}$} }
 \end{figure}

\begin{theorem}\label{positive 2}
Let $G_w$ be a weighted unicyclic graph of order $n$. Then
$i_+(G_w)=2$ if and only if $G_w$ is one of the following graphs:
$C_3^w$ of Type C; $C_4^w$  of  Type $B$; $C_5^w$ of Type D; $C_6^w$
 of  Type $A$; the weighted graphs with $U_1^{r,s}$ or
$U_2^{p,q}$ as the underlying graph; the weighted graphs with
$U_3^{n-4}$ as the underlying graph in which the cycle $C_3^w$ is of
Type $D$; the weighted graphs with $U_4^{n-5}$ as the underlying
graph in which the cycle $C_4^w$ is of  Type $A$.
\end{theorem}
\begin{proof}
The sufficiency can be easily verified by Lemmas \ref{deleting pendent
vertex} and \ref{cycle}. Next we consider the necessity. Assume that
the girth of $G_w$ is $k$.

If $k=n$, by virtue of Lemma \ref{cycle}, $G_w$ is one of the
following weighted cycles: $C_3^w$  of Type $C$; $C_4^w$  of Type
$B$; $C_5^w$ of Type $D$; $C_6^w$  of Type $A$.

If $k=n-1$, by Lemmas \ref{deleting pendent vertex} and
\ref{weighted tree inertia}, $G_w$ is one of the weighted graphs
with $U_1^{1,0}$ or $U_2^{1,0}$ as the underlying graph.

If $3\leq k\leq n-2$, then by Theorem~\ref{thm3-3}, $2=i_{+}(G_w)\ge \lceil\frac{k}{2}\rceil $, which implies that $k=3$ or $4$ and $i_{+}(G_w)= \lceil\frac{k}{2}\rceil $. Hence it follows from  Theorem \ref{extremal graphs1} that (1) or (2) in Theorem~\ref{extremal graphs1} holds. We consider the following two cases.
If (1) in Theorem~\ref{extremal graphs1} holds, i.e, there exists a vertex $v_i\in V(C_k^w)$ which is saturated in $G_w\{v_i\}$, then $G_w\{v_i\}$ is  a star and $m(G_w-G_w\{v_i\}=\lfloor\frac{k-1}{2}\rfloor=1$, i.e,  $G_w-G_w\{v_i\}$ is a star.
 Hence  $G_w$ is one of the
following graphs: the weighted graphs with $U_1^{r,s}$ or
$U_2^{p,q}$ as the underlying graph; the weighted graphs with
$U_3^{n-4}$ as the underlying graph in which the cycle $C_3^w$ is of
Type $D$; the weighted graphs with $U_4^{n-5}$ as the underlying
graph in which the cycle $C_4^w$ is of Type $A$.

 If (2) in Theorem~\ref{extremal graphs1} holds, i.e,
 there does not exist a vertex $v_i\in V(C_k^w)$ which is saturated in $G_W\{v\}$, then
 $G_w$  is $U_{3}^{n-4}$  and $C_3^w$ is Type $D$, or $G_w$ is
 $U_{4}^{n-5}$ and $C_{4}^w$ is the Type A. Hence the assertion
 holds.
\end{proof}

Similar to the above result, we have
\begin{theorem}\label{negative 2}
Let $G_w$ be a weighted unicyclic graph of order $n$. Then
$i_-(G_w)=2$ if and only if $G_w$ is one of the following graphs:
$C_3^w$ of Type D; $C_4^w$  of Type B; $C_5^w$ of Type C; $C_6^w$ of
Type A; the weighted graphs with $U_1^{r,s}$ or $U_2^{p,q}$ as the
underlying graph; the weighted graphs with $U_3^{n-4}$ as the
underlying graph in which the cycle $C_3^w$ is of Type C; the
weighted graphs with $U_4^{n-5}$ as the underlying graph in which
the cycle $C_4^w$ is of Type A.
\end{theorem}

Note that there does not exist weighted unicyclic graph $G_w$ with
$i_+(G_w)=1$ and $i_-(G_w)=3$, or $i_+(G_w)=3$ and $i_-(G_w)=1$ from
Lemma \ref{difference}. Combining this fact with Theorems
\ref{positive 2} and \ref{negative 2}, we have

\begin{theorem}\label{nullity n-4}
Let $G_w$ be a weighted unicyclic graph of order $n$. Then
$i_0(G_w)=n-4$ if and only if $G_w$ is one of the following graphs:
$C_4^w$  of Type B; $C_6^w$  of Type A; the weighted graphs with
$U_1^{r,s}$ or $U_2^{p,q}$ as the underlying graph; the weighted
graphs with $U_4^{n-5}$ as the underlying graph in which the cycle
$C_4^w$ is of Type A.
\end{theorem}

 From Theorem
\ref{nullity n-4}, we  get that

\begin{corollary} \cite{fan1}
Let $\Gamma$ be a unicyclic signed graph of order $n$. Then
$i_0(\Gamma)=n-4$ if and only if $\Gamma$ is one of the following
signed graphs of order $n$: unbalance $C_4$; unbalance $C_6$; the
signed graphs with $U_1^{r,s}$ or $U_2^{p,q}$ as the underlying
graph; the balance signed graph with $U_4^{n-5}$ as the underlying
graph.
\end{corollary}

\begin{corollary} \cite{tan}
Let $U$ be an unweighted unicyclic graph of order $n$. Then
$i_0(U)=n-4$ if and only if $U$ is one of the following graphs of
order $n$: $U_1^{r,s}$, $U_2^{p,q}$, or $U_4^{n-5}$.
\end{corollary}

\section{Weighted unicyclic graphs with rank 6}

\begin{lemma}\label{girth}
Let $G_w$ be a weighted unicyclic graph with three positive
(negative) eigenvalues and girth $k$. Then $k\leq 8$.
\end{lemma}
\begin{proof}
If $k\geq 9$, then $G_w$ must contain $P^w_8$ as an induced
subgraph. By Lemma \ref{induced subgraph}, $i_+(G_w)\geq 4$
$(i_-(G_w)\geq 4)$ which is a contradiction.
\end{proof}

In what follows we shall characterize the weighted unicyclic graphs
with three positive eigenvalues. Let $\mathcal {U}^*$ be the set of
weighted unicyclic graphs without pendant twins. By Lemma
\ref{pendant twin}, it suffices to characterize the weighted
unicyclic graphs among all graphs in $\mathcal {U}^*$. Let $G_1$
(resp. $G_2$) be the unweighted graph obtained from $C_8$ (resp.
$C_7$) by attaching a pendant edge on a vertex of $C_8$ (resp.
$C_7$).
\begin{theorem}\label{eight}
Let $G_w\in \mathcal {U}^*$ be a weighted unicyclic graph with girth
$k$. Then
\begin{enumerate}[(1).]
\item If $k=7$, $i_+(G_w)=3$ if and only if $G_w$ is $C_7^w$ of Type
$D$.
\item  If $k=8$, $i_+(G_w)=3$ if and only if $G_w$ is  $C_8^w$ of Type
$A$.
\end{enumerate}
\end{theorem}
\begin{proof}
The sufficiency can be  easily verified by Lemma \ref{cycle}. Next we
consider the necessity.

 If $G_w$ is a cycle, by Lemma \ref{cycle}
$G_w$ is $C_8^w$ of Type $A$, or $C_7^w$ of Type $D$. Assume that
the underlying graph of $G_w$ contains $G_1$ or $G_2$ as an induced
subgraph. By Lemma \ref{deleting pendent vertex}, we have
$i_+(G_w)\geq i_+(G_1)=1+i_+(P_7)=4$ and $i_+(G_w)\geq
i_+(G_2)=1+i_+(P_6)=4$ which contradicts the fact that $i_+(G_w)=3$.
This completes the proof.
\end{proof}

\begin{figure}[ht]
\center
\includegraphics [width = 8cm]{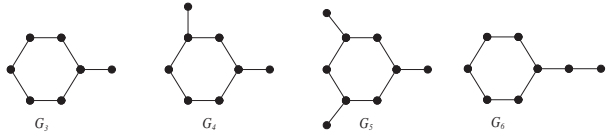}
\caption { \textit{ Four unweighted graphs in Theorem 4.3} }
 \end{figure}

\begin{theorem}\label{six}
Let $G_w\in \mathcal {U}^*$ be a weighted unicyclic graph with girth
$6$. Then $i_+(G_w)=3$ if and only if $G_w$ is one of the following
graphs: the weighted graphs with $G_3$, $G_4$ or $G_5$ (as depicted
in Fig. 2) as the underlying graph; the weighted graphs with $G_6$
(as depicted in Fig. 2) as the underlying graph in which the cycle
$C_6^w$ is of Type $A$.
\end{theorem}

\begin{figure}[ht]
\center
\includegraphics [width = 6cm]{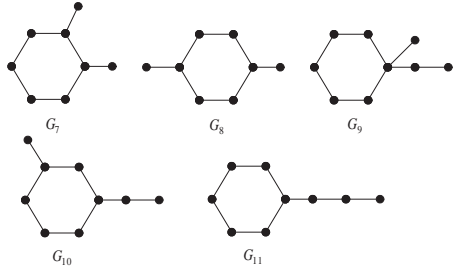}
\caption { \textit{ Five unweighted graphs excluded by $i_+(G_w)=3$}
}
 \end{figure}
\begin{proof}
The sufficiency can be easily verified by Lemmas \ref{deleting pendent
vertex} and \ref{cycle}. Next we consider the necessity.

By Lemmas \ref{deleting pendent vertex} and \ref{cycle}, the graphs
with one of $G_i$'s ($i=7,8,\cdots,11$) (as depicted in Fig. 3) as
the underlying graph have four positive eigenvalues. Let $G_w\in
\mathcal {U}^*$ be a weighted graph with girth $6$ and positive
index 3.

\textbf{Case 1.} $G_w-C_6^w$ is a set of isolated vertices.

If the order of the graph $G_w-C_6^w$ is 1, $G_w$ is the graph with
$G_3$ as the underlying graph.

If the order of the graph $G_w-C_6^w$ is 2, $G_w$ is the graph with
$G_4$ as the underlying graph. The graph with $G_7$ or $G_8$ as the
underlying graph has four positive eigenvalues.

If the order of the graph $G_w-C_6^w$ is at least 3, $G_w$ is the
graph with $G_5$ as the underlying graph. Moreover, other graphs
must contain $G_7$ as an induced subgraph and have more than three
positive eigenvalues.

\textbf{Case 2.} $G_w-C_6^w$ contains $P_2$ as an induced subgraph.

Then $G_w$ is the graph with $G_6$ as the underlying graph in which
the cycle $C_6^w$ is of Type $A$. Any other graph has more than
three positive eigenvalues since its underlying graph contain one of
$G_i$'s ($i=7,8,\cdots,11$) as an induced subgraph.
\end{proof}

Similar to Theorem \ref{six}, we have
\begin{theorem}\label{five}
Let $G_w\in \mathcal {U}^*$ be a weighted unicyclic graph with girth
$5$. Then $i_+(G_w)=3$ if and only if $G_w$ is one of the following
graphs: the weighted graph with one of $G_{i}$'s $(i=12,\cdots,15)$
(as depicted in Fig. 4) as the underlying graph; the weighted graph
with $G_{16}$ (as depicted in Fig. 4) as the underlying graph in
which the cycle $C_5^w$ is of Type $D$.
\end{theorem}

\begin{figure}[ht]
\center
\includegraphics [width = 10cm]{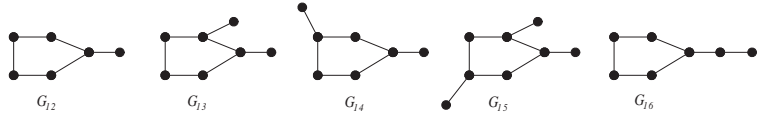}
\caption { \textit{ Five unweighted graphs in Theorem 4.4} }
 \end{figure}

\begin{theorem}\label{four}
Let $G_w\in \mathcal {U}^*$ be a weighted unicyclic graph with girth
$4$. Then $i_+(G_w)=3$ if and only if $G_w$ is one of the following
graphs: the weighted graphs with one of $G_i$'s ($i=18,
19,\cdots,28$) (as depicted in Fig. 5) as the underlying graph; the
weighted graphs with $G_{17}$ (as depicted in Fig. 5) as the
underlying graph in which the cycle $C_4^w$ is of Type $B$; the
weighted graphs with one of $G_i$'s ($i=29,30,\cdots,34$) (as
depicted in Fig. 5) as the underlying graph in which the cycle
$C_4^w$ is of Type $A$.
\end{theorem}
\begin{figure}[ht]
\center
\includegraphics [width = 10cm]{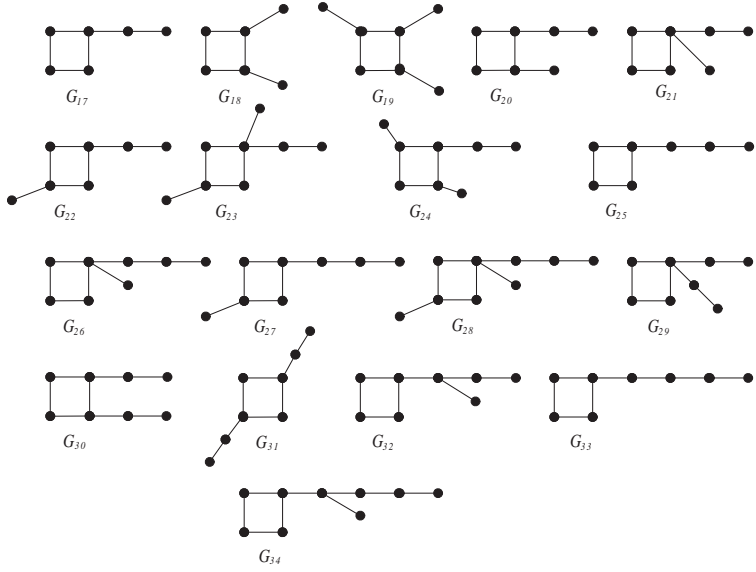}
\caption { \textit{ Eighteen unweighted graphs in Theorem 4.5} }
 \end{figure}

\begin{figure}[ht]
\center
\includegraphics [width = 10cm]{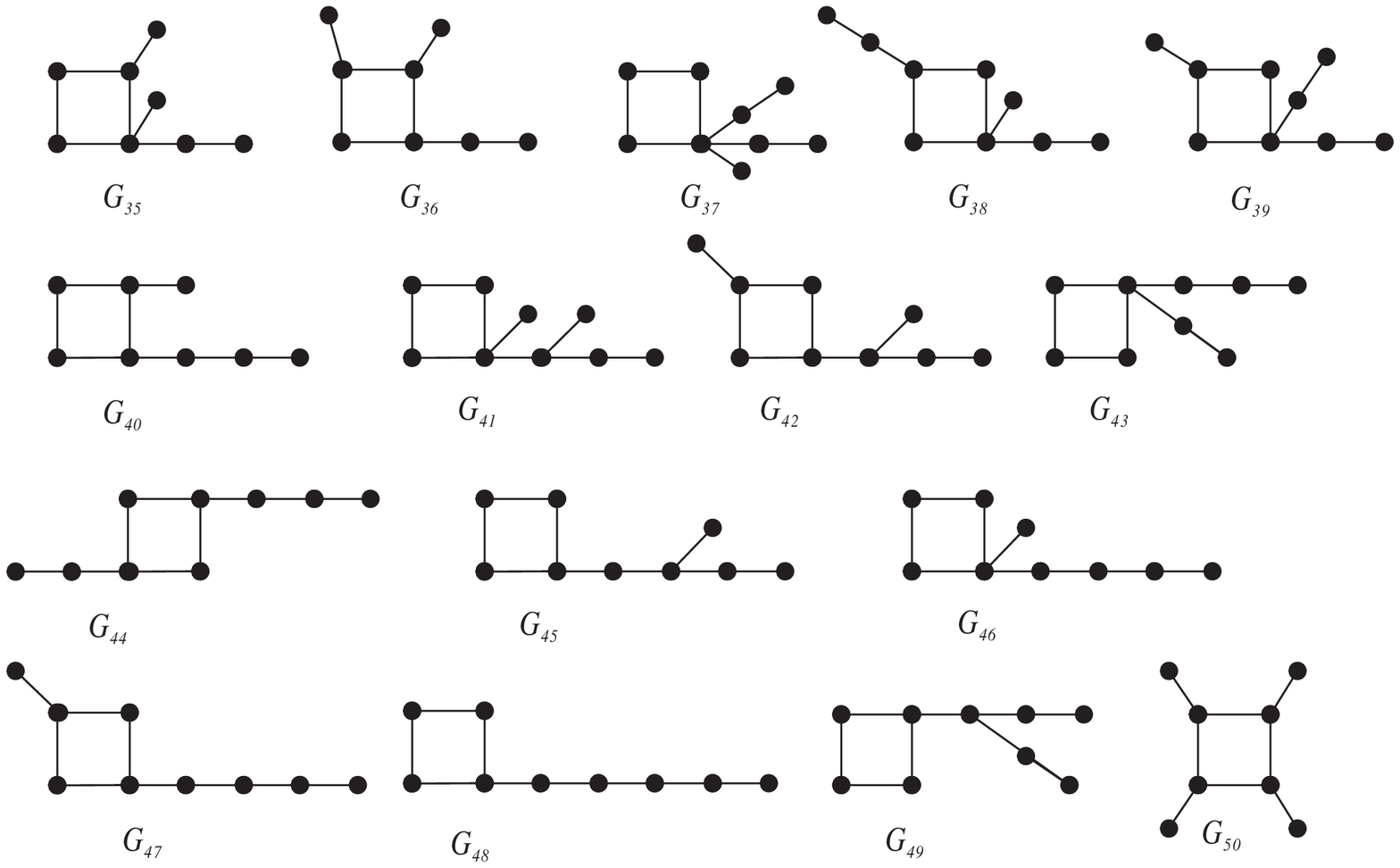}
\caption { \textit{ Sixteen unweighted graphs excluded by
$i_+(G_w)=3$} }
 \end{figure}

 \begin{proof}
The sufficiency is easily verified by Lemmas \ref{deleting pendent
vertex} and \ref{cycle}. Next we consider the necessity.

By Lemmas \ref{deleting pendent vertex} and \ref{cycle}, the graphs
with one of $G_i$'s ($i=35,36,\cdots,50$) (as depicted in Fig. 6) as
the underlying graph have four positive eigenvalues. Let $G_w\in
\mathcal {U}^*$ be a weighted graph with girth $4$ and positive
index 3. For convenience, denote by $G^\diamondsuit$
 the underlying graph of
$G_w-C_4^w$.

\textbf{Case 1.} $G^\diamondsuit$ is a set of isolated vertices.

$G_w$ is the weighted graph with $G_{18}$, or $G_{19}$ as the
underlying graph.

\textbf{Case 2.} $G^\diamondsuit$ contains $P_2$, but  no
$P_3$, as an induced subgraph .

If $G^\diamondsuit=P_2$, $G_w$ is the weighted graph with $G_{17}$
as the underlying graph in which the cycle $C_4^w$ is of Type $B$.

If $G^\diamondsuit$ is the union of an isolated vertex and $P_2$,
$G_w$ is the weighted graph with $G_{20}$, $G_{21}$ or $G_{22}$ as
the underlying graph.

If $G^\diamondsuit$ is the union of two isolated vertices and $P_2$,
 $G_w$ is the weighted graph
with $G_{23}$ or $G_{24}$ as its underlying graph.

If $G^\diamondsuit$ is the union of more than two isolated vertices
and $P_2$, any weighted graph has more than three positive
eigenvalues since its underlying graph contains $G_{35}$ or $G_{36}$
as an induced subgraph.

If $G^\diamondsuit$ is two copies of $P_2$, $G_w$ is the weighted
graph with $G_{29}$, $G_{30}$ or $G_{31}$ as the underlying graph in
which the cycle $C_4^w$ is of Type $A$.

If $G^\diamondsuit$ is the union of some isolated vertices and two
$P_2$'s, any weighted graph has more than three positive eigenvalues
since its underlying graph contains one of $G_i$'s
($i=35,36,\cdots,39$) as an induced subgraph.

If $G^\diamondsuit$ contains three $P_2$'s as its induced subgraph,
 any weighted graph has more than
three positive eigenvalues since its underlying graph contains one
of $G_i$'s ($i=35,36,37$) as an induced subgraph.

\textbf{Case 3.} $G^\diamondsuit$ contains $P_3$, but no
$P_4$, as an induced subgraph.

If $G^\diamondsuit=P_3$, $G_w$ is the weighted graph with $G_{25}$
as its underlying graph.

If $G^\diamondsuit$ is the union of one isolated vertex and $P_3$,
$G_w$ is one of the following graphs: the weighted graphs with
$G_{26}$ or $G_{27}$ as the underlying graph.

If $G^\diamondsuit$ is the union of two isolated vertices and $P_3$,
$G_w$ is the weighted graph with $G_{28}$ as the underlying graph.

If $G^\diamondsuit$ is the union of more than two isolated vertices
and $P_3$, any graph has more than three positive eigenvalues since
it contains $G_{40}$, $G_{41}$ or $G_{42}$ as an induced subgraph.

If $G^\diamondsuit$ contains the union of $P_2$ and $P_3$ as its
induced subgraph, any graph has more than three positive eigenvalues
since it contains $G_{40}$, $G_{43}$ or $G_{44}$ as an induced
subgraph.

\textbf{Case 4.} $G^\diamondsuit$ contains $P_4$, but no
$P_5$, as an induced subgraph.

$G_w$ is the weighted graph with $G_{32}$, or $G_{33}$ as the
underlying graph in which the cycle $C_4^w$ is of Type $A$. All
other graphs have more than three positive eigenvalues since their
underlying graphs contain one of $G_{i}$'s ($i=40,41,\cdots,47$) as
an induced subgraph.

\textbf{Case 5.} $G^\diamondsuit$ contains $P_5$ as an induced
subgraph.

$G_w$ is the weighted graph with $G_{34}$ as the underlying graph in
which the cycle $C_4^w$ is of Type $A$. All other graphs have more
than three positive eigenvalues since their underlying graph contain
one of $G_{i}$'s ($i=40,41,\cdots,49$) as an induced subgraph.
\end{proof}

Similar to Theorem \ref{four}, we have
 \begin{theorem}\label{three}
Let $G_w\in \mathcal {U}^*$ be a weighted unicyclic graph with girth
$3$. Then $i_+(G_w)=3$ if and only if $G_w$ is one of the following
graphs: the weighted graphs with one of $G_i$'s
($i=52,53,\cdots,60$) (as depicted in Fig. 7) as the underlying
graph; the weighted graphs with $G_{51}$ (as depicted in Fig. 7) as
the underlying graph in which the cycle $C_3^w$ is of Type $C$; the
weighted graphs with one of $G_{i}$'s ($i=61,62,\cdots,65$) (as
depicted in Fig. 7) as the underlying graph in which the cycle
$C_3^w$ is of Type $D$.
\end{theorem}
\begin{figure}[ht]
\center
\includegraphics [width = 12cm]{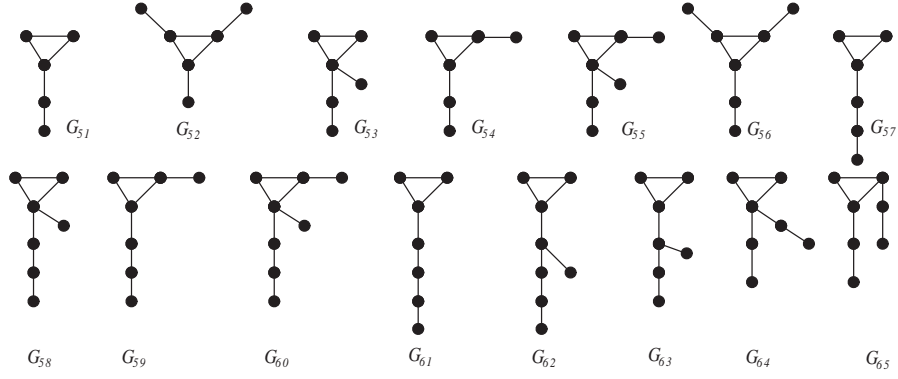}
\caption { \textit{ Fifteen unweighted graphs in Theorem 4.6} }
 \end{figure}

 \begin{theorem}\label{rank 6}
Let $G_w\in \mathcal {U}^*$ be a weighted unicyclic graph with rank
$6$ and girth $k$. Then $k=3,4,5,6$, or $8$ and
 \begin{enumerate}[(1).]
\item if $k=3$, then $G_w$ is one of the weighted graphs with one of $G_i$'s
($i=52,53,\cdots,60$) (as depicted in Fig. 7) as the underlying
graph;
\item if $k=4$, then $G_w$ is one of the following graphs: the weighted graph with one of $G_i$'s ($i=18,
19,\cdots,28$) (as depicted in Fig. 5) as the underlying graph; the
weighted graph with $G_{17}$ (as depicted in Fig. 5) as the
underlying graph in which the cycle $C_4^w$ is of Type $B$; the
weighted graph with one of $G_i$'s ($i=29,\cdots,34$) (as depicted
in Fig. 5) as the underlying graph in which the cycle $C_4^w$ is of
Type $A$;
\item if $k=5$, then $G_w$ is one of the weighted graphs with one of $G_{i}$'s $(i=12,\cdots,15)$ (as depicted in Fig. 4) as the underlying
graph;
\item if $k=6$, then $G_w$ is one of the following graphs: the weighted graph with $G_3$, $G_4$ or $G_5$ (as depicted
in Fig. 2) as the underlying graph; the weighted graph with $G_6$
(as depicted in Fig. 2) as the underlying graph in which the cycle
$C_6^w$ is of Type A;
\item if $k=8$, then $G_w$ is $C_8^w$ of Type $A$.
\end{enumerate}
 \end{theorem}

 \begin{proof}
By Lemma \ref{difference}, $r(G_w)=6$ if and only if
$i_+(G_w)=i_-(G_w)=3$ for any weighted unicyclic $G_w$. So it
suffices to characterize the weighted unicyclic graphs with rank $6$
among all weighted unicyclic graphs with three positive eigenvalues.
That is to say, we eliminate graphs $G_w$ with $i_-(G_w)\neq 3$
among all graphs with three positive eigenvalues. So the results
follow from Lemma \ref{girth} and Theorems \ref{eight}--\ref{three}.
 \end{proof}
From Theorem \ref{rank 6} and Lemma \ref{cycle}, we have
 \begin{corollary}\label{uniweighted rank 6}
Let $G\in \mathcal {U}^*$ be an unweighted unicyclic graph with rank
$6$ and girth $k$. Then $k=3,4,5,6$,or $8$ and
 \begin{enumerate}[(1).]
\item if $k=3$, then $G$ is one of $G_i$'s
($i=52,53,\cdots,60$) (as depicted in Fig. 7);
\item if $k=4$, then $G$ is one of $G_i$'s ($i=18,
19,\cdots,34$) (as depicted in Fig. 5);
\item if $k=5$, then $G$ is one of $G_{i}$'s $(i=12,\cdots,15)$ (as depicted in Fig. 4);
\item if $k=6$, then $G$ is one of $G_i$'s  $(i=3,4,5)$ (as depicted
in Fig. 2);
%\item If $k=7$, then $G_w$ is $C_7^w$ of Type $D$;
\item if $k=8$, then $G$ is $C_8$.
\end{enumerate}
 \end{corollary}

 \begin{corollary}\label{signed rank 6}
Let $\Gamma\in \mathcal {U}^*$ be a unicyclic signed graph with rank
$6$ and girth $k$. Then $k=3,4,5,6$, or $8$ and
 \begin{enumerate}[(1).]
\item if $k=3$, then $\Gamma$ is one of the signed graph with one of $G_i$'s
$(i=52,53,\cdots,60)$ (as depicted in Fig. 7) as the underlying
graph;
\item if $k=4$, then $\Gamma$ is one of the following graphs: the signed graph with one of $G_i$'s $(i=18,
19,\cdots,28)$ (as depicted in Fig. 5) as the underlying graph; the
unbalanced signed graph with $G_{17}$ (as depicted in Fig. 5) as the
underlying graph; the balanced signed graph with one of $G_i$'s
$(i=29,\cdots,34)$ (as depicted in Fig. 5) as the underlying graph;
\item if $k=5$, then $\Gamma$ is one of the signed graphs with $G_{12}$, $G_{13}$,
$G_{14}$ or $G_{15}$ (as depicted in Fig. 4) as the underlying
graph;
\item if $k=6$, then $\Gamma$ is one of the following graphs: the signed graph with $G_3$, $G_4$ or $G_5$ (as depicted
in Fig. 2) as the underlying graph; the balanced signed graph with
$G_6$ (as depicted in Fig. 2) as the underlying graph;
\item if $k=8$, then $\Gamma$ is the balanced cycle $C_8$.
\end{enumerate}
 \end{corollary}

 By Lemma \ref{pendant twin} and Theorem \ref{rank 6}, we have
 \begin{theorem}\label{rank6 n}
Let $G_w$ be a weighted unicyclic graph with rank $6$. Then the
underlying graph of $G_w$ must contain one of the graphs described
in Theorem \ref{rank 6} as an induced subgraph.
 \end{theorem}

 By Lemma \ref{pendant twin} and Theorem \ref{rank6 n}, we can characterize all weighted (unweighted)
 unicyclic graphs of order $n$ with rank $6$ by attaching some appropriate
 pendant vertices to certain neighbors of a few pendant vertices of the graphs described in Corollary \ref{uniweighted rank 6}.
 \begin{example}
All unweighted unicyclic graphs of order $8$ with rank $6$ are $C_8$
and the 45 graphs depicted in Fig. 8. This result corresponds to the
one obtained by Cvetkovi\'{c} and Rowlinson \cite{cvetkovic}.
 \end{example}
 \begin{figure}[ht]
\center
\includegraphics [width = 10cm]{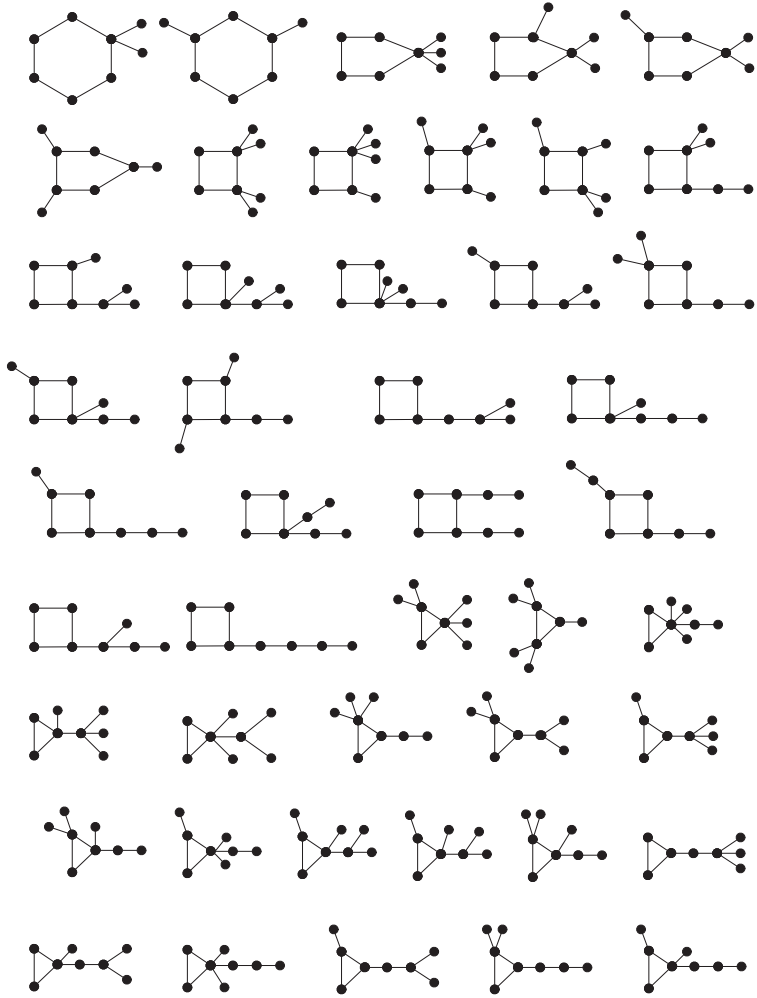}
\caption { \textit{ 45 unweighted unicyclic graphs on eight
vertices} }
 \end{figure}

 \section{Weighted unicyclic graphs with rank $2$, $3$, $5$}

Fan et al \cite{fan1} characterized the unicyclic signed graph of
order $n$ with rank $2$, $3$, $4$, $5$, respectively. In this
section we shall consider the same question following the ideas in
\cite{fan1}. In Section 3, we characterized the weighted unicyclic
graphs of order $n$ with rank $4$. Here we shall determine the
weighted unicyclic graphs of order $n$ with rank $2$, $3$, $5$,
respectively.

\begin{theorem}\label{rank 2 3}
Let $G_w$ be a weighted unicyclic graph of order $n$ and $C_k^w$ be the
unique cycle in $G_w$ with vertex set $\{v_1,v_2,\cdots,v_k\}$. Then
%\begin{enumerate}
%\item

(1). $i_0(G_w)=n-2$ ($r(G_w)=2$) if and only if $G_w$ is the
weighted cycle $C_4$ which is of Type $A$.

(2). $i_0(G_w)=n-3$ ($r(G_w)=3$) if and only if $G_w$ is the cycle
$C_3$ with arbitrary weights.
%\end{enumerate}
\end{theorem}
\begin{proof}
It is obvious that the sufficiency for (1) or (2) holds by Lemma
\ref{cycle}. Next we consider the necessity.

Assume that $i_0(G_w)=n-2$. If $G_w$ is a weighted cycle, then $G_w$
is the weighted cycle $C_4$ which is of Type $A$ by Lemma
\ref{cycle}. Next assume that $G_w$ contains pendant edges. Suppose
there exists a vertex $v_i$ in $C_k^w$ such that it is saturated in
$G_w\{v_i\}$. Without loss of generality, suppose $v_1 \in V(C_k^w)$
is saturated in $G_w\{v_1\}$. By Lemma \ref{inertia of unicyclic
graph}, we have
\begin{eqnarray*}
i_0(G_w)&=&i_0(G_w\{v_1\})+i_0(G_w-G_w\{v_1\})\\
&=&n-2m(G_w\{v_1\})-2m(G_w-G_w\{v_1\}).
\end{eqnarray*}
Since $m(G_w\{v_1\})\geq 1$ and $m(G_w-G_w\{v_1\})\geq 1$,
$i_0(G_w)\leq n-4$. This is a contradiction.

Suppose there does not exist a vertex $v_i\in V(C_k^w)$ which
is saturated in $G_w\{v_i\}$. By Lemma \ref{inertia of unicyclic
graph},
\begin{eqnarray*}
i_0(G_w)&=&i_0(G_w-C_k^w)+i_0(C_k^w)\\
&=&n-k-2m(G_w-C_k^w)+i_0(C_k^w).
\end{eqnarray*}
It yields that $i_0(C_k^w)=k+2\big(m(G_w-C_k^w)-1\big)\geq k\geq3$
which is a contradiction.

Assume that $i_0(G_w)=n-3$. If $G_w$ is a weighted cycle, by Lemma
\ref{cycle}, $G_w$ is $C_3^w$ with arbitrary weights. Assume that
$G_w$ contains at least one pendant edge. By the above discussion,
if there exists a vertex $v_i\in V(C_k^w)$ which is saturated in
$G_w\{v_i\}$, then $i_0(G_w)\leq n-4$ which is a contradiction. If
there does not exist a vertex $v_i\in V(C_k^w)$  which is saturated
in $G_w\{v_i\}$, then $i_0(C_k^w)=k+2m(G_w-C_k^w)-3\geq k-1\geq 2$.
So only if $k=3$, $i_0(G_w)=2$. But $i_0(C_3^w)=0$. This case cannot
occur.
\end{proof}

The following results are immediate from Theorem \ref{rank 2 3}.
\begin{corollary}
Let $G$ be an unweighted unicyclic graph of order $n$. Then
\begin{enumerate}[(1).]
\item $i_0(G)=n-2$ ($r(G)=2$) if and only if $G$ is
the cycle $C_4$.
\item $i_0(G)=n-3$ ($r(G)=3$) if and only if $G$ is
the cycle $C_3$.
\end{enumerate}
\end{corollary}

\begin{corollary} \cite{fan1}
Let $\Gamma$ be a unicyclic signed graph of order $n$. Then
\begin{enumerate}[(1).]
\item $i_0(\Gamma)=n-2$ ($r(\Gamma)=2$) if and only if $\Gamma$ is
the balanced cycle $C_4$.
\item $i_0(\Gamma)=n-3$ ($r(\Gamma)=3$) if and only if $\Gamma$ is
the cycle $C_3$.
\end{enumerate}
\end{corollary}

Let $H_{n,3}^1$ be an unweighted unicyclic graph obtained by joining
a vertex of $C_3$ and the center of $K_{1, n-4}$, the star of order
$n-3$.

\begin{theorem}\label{rank 5}
Let $G_w$ be a weighted unicyclic graph of order $n\geq 5$. Then
$i_0(G_w)=n-5$ ($r(G_w)=5$) if and only if $G_w$ is the weighted
graph with $C_5$, or $H_{n,3}^1$ as the underlying graph.
\end{theorem}
\begin{proof}
It is obvious that the sufficiency holds by Lemmas \ref{cycle} and
\ref{deleting pendent vertex}.

Necessity: If $G_w$ is a weighted cycle, by Lemma \ref{cycle}, $G_w$
is $C_5^w$ with arbitrary weights. Next assume that $G_w$ contains
at least one pendant edge. Suppose there exists a vertex $v_i\in
V(C_k^w)$ which is saturated in $G_w\{v_i\}$. Then
$i_0(G_w)=n-2m(G_w\{v_i\})-2m(G_w-G_w\{v_i\})$. So
$2m(G_w\{v_i\})+2m(G_w-G_w\{v_i\})=5$ which is a contradiction.

Assume that for any vertex $v\in V(C_k^w)$ it is not saturated in
$G_w\{v\}$. By Lemma \ref{cycle}, we have
$i_0(G_w)=n-k-2m(G_w-C_k^w)+i_0(C_k^w).$ Hence
$$i_0(C_k^w)=k+2m(G_w-C_k^w)-5.\eqno(*)$$

Note that $i_0(C_k^w)=0$, or $2$. If $i_0(C_k^w)=0$, then $k=3$ and
$m(G_w-C_k^w)=1$ which implies $G_w-C_k^w$ is a star. Hence the
underlying graph $G$ of $G_w$ is isomorphism to $H_{n,3}^1$. If
$i_0(C_k^w)=2$,  we have $k\leq 5$. Then $k=4$ by Lemma \ref{cycle}.
From $(*)$, $2m(G_w-C_k^w)=3$ which is a contradiction. This case
cannot hold.
\end{proof}

The next results follow from Theorem \ref{rank 5}.
\begin{corollary}\cite{guo}
Let $G$ be a unicyclic graph of order $n$. Then $i_0(G)=n-5$
($r(G)=5$) if and only if $G$ is the cycle $C_5$ or $H_{n,3}^1$.
\end{corollary}

\begin{corollary}\cite{fan1}
Let $\Gamma$ be a unicyclic signed graph of order $n$. Then
$i_0(\Gamma)=n-5$ ($r(\Gamma)=5$) if and only if $\Gamma$ is the
following graphs: the cycle $C_5$; the signed graph with $H_{n,3}^1$
as the underlying graph.
\end{corollary}

{\bf Acknowledgement:} The authors are grateful to two anonymous
referees for many helpful comments and suggestion to an earlier version of this
paper.

\end{document}